\documentclass[amstex,12pt]{amsart}
\usepackage{mathrsfs}
\usepackage{amssymb}
\usepackage{amsfonts}
\usepackage{amsfonts,amsmath,amsthm, amssymb}
\usepackage{array}
\usepackage{latexsym, euscript, epic, eepic}
\usepackage{tikz}
\usetikzlibrary{matrix}

\newtheorem{theorem}{Theorem}[section]
\newtheorem{lemma}[theorem]{Lemma}
\newtheorem{proposition}[theorem]{Proposition}
\newtheorem{corollary}[theorem]{Corollary}
\newtheorem{remark}[theorem]{Remark}
\newtheorem*{definition}{Definition}

\begin{document}

\title[Teich\-m\"ul\-ler spaces of piecewise symmetric homeomorphisms]
{Teich\-m\"ul\-ler spaces of piecewise symmetric homeomorphisms on the unit circle} 

\author[H. Wei]{Huaying Wei} 
\address{Department of Mathematics and Statistics, Jiangsu Normal University \endgraf Xuzhou 221116, PR China} 
\email{hywei@jsnu.edu.cn} 

\author[K. Matsuzaki]{Katsuhiko Matsuzaki}
\address{Department of Mathematics, School of Education, Waseda University \endgraf
Shinjuku, Tokyo 169-8050, Japan}
\email{matsuzak@waseda.jp}

\thanks{Research supported by the National Natural Science Foundation of China (Grant No. 11501259)
and Japan Society for the Promotion of Science (KAKENHI 18H01125).}

\begin{abstract}
We interpolate a new family of Teichm\"uller spaces $T_{\sharp}^X$ between 
the universal Teichm\"uller space $T$ and its little subspace $T_0$, which 
we call the Teichm\"uller space of piecewise symmetric homeomorphisms. This is defined by prescribing a subset $X$ of the unit circle.
The inclusion relation of $X$ induces a natural inclusion of $T_{\sharp}^X$, and
an approximation of $T$ is given by an increasing sequence of $T_{\sharp}^X$.
In this paper, we discuss the fundamental properties of $T_{\sharp}^X$ from the viewpoint of the quasiconformal 
theory of Teichm\"uller spaces. We also consider the quotient space of $T$ by $T_{\sharp}^X$ as an analog 
of the asymptotic Teichm\"uller space. 
%Finally, we show its close relationships with $T$ and $T_0$. 
\end{abstract}

\maketitle

\section{Introduction}
A sense-preserving self-homeomorphism $h$ of the unit circle $\mathbb{S} = \{z\in\mathbb{C} \mid  |z|=1\}$ is said to be {\it quasisymmetric} if there exists a (least) positive constant $C(h)$, called the quasisymmetry constant of $h$, such that 
$$
\frac{\mid h(I_1)\mid}{\mid h(I_2) \mid}\leqslant C(h)
$$
for all pairs of adjacent intervals $I_1$ and $I_2$ on $\mathbb{S}$ with the same length $|I_1|=|I_2|$. 
Beurling and Ahlfors \cite{BA} proved that a sense-preserving self-homeomorphism $h$ of $\mathbb{S}$ is quasisymmetric if and only if there exists some quasiconformal homeomorphism of the unit disk $\mathbb{D} = \{z\in\mathbb{C} \mid  |z|<1\}$ onto itself that has boundary value $h$. Later, Douady and Earle \cite{DE} gave a quasiconformal extension of a quasisymmetric homeomorphism of $\mathbb{S}$, called the barycentric extension, in a conformally invariant way. 

The {\it universal Teich\-m\"ul\-ler space} $T$ is a universal parameter space of marked complex structures on
all Riemann surfaces and can be defined as the group $\rm QS$ of all quasisymmetric homeomorphisms of 
$\mathbb{S}$ modulo the left action of the group 
$\mbox{\rm M\"ob}(\mathbb{S})$ of all M\"obius transformations of $\mathbb{S}$, i.e., 
$T= \mbox{\rm M\"ob}(\mathbb{S}) \backslash \rm QS$. The quotient by $\mbox{\rm M\"ob}(\mathbb{S})$ is alternatively
achieved by giving a normalization to elements in $\rm QS$.
A topology of $T$ is induced by the quasisymmetry constants of normalized quasi\-symmetric homeo\-morphisms. 
It is known that $T$ is contractible (see \cite{EE} and also \cite{DE}) and is an infinite dimensional complex manifold modeled 
on a certain Banach space via the Bers embedding through the Schwarzian derivative (see \cite{GL, Le, Na}). 

A quasisymmetric homeomorphism $h$ is called {\it symmetric} if 
$$
\frac{\mid h(I_1)\mid}{\mid h(I_2) \mid}\to 1
$$
uniformly as $|I_1|=|I_2|\to 0$. 
Let $\rm Sym$ denote the set of all symmetric homeomorphisms of $\mathbb{S}$.
It is known that $h$ is symmetric if and only if $h$ can be extended to an {\it asymptotically conformal} homeomorphism $f$ of $\mathbb{D}$ onto itself in the sense that its complex dilatation 
$\mu = \bar{\partial}f/\partial f$ vanishes at the boundary. 
This result is attributed to Fehlmann \cite{Fe} in \cite{GS}. It is proved by Earle, Markovic, and Saric \cite{EMS} that the barycentric extension of a symmetric homeomorphism $h$ is asymptotically conformal. 
We denote $\mbox{\rm M\"ob}(\mathbb{S}) \backslash \rm Sym$ by $T_0$ and call it the little universal Teichm\"uller space. 

This little subspace $T_0$ of $T$ as well as 
the asymptotic Teichm\"uller space $T_0\backslash T$ was investigated in depth by Gardiner and Sullivan \cite{GS}. In particular, they endowed $T_0$ with a complex Banach manifold structure via the Bers embedding, and proved that the Bers embedding is compatible with the coset decomposition $T_0\backslash T$ and the quotient of the Banach spaces. In particular, $T_0\backslash T$ is equipped with
the complex structure modeled on the quotient Banach space.

We can localize the definition of symmetric homeomorphism.
We say that a quasisymmetric homeomorphism
$h$ is {\it symmetric on the closed interval} $I$ of $\mathbb{S}$ 
if the above uniform convergence $|h(I_1)|/|h(I_2)| \to 0$ holds for all allowable intervals $I_1$ and $I_2$ in $I$.
It was shown by Fehlmann \cite{Fe} (also see \cite[Proposition 3.1]{GS}) that $h$ is symmetric on $I$ if and only if $h$ has a local dilatation,
which is the infimum of maximal dilatations of any possible local quasiconformal extensions of
$h$, equal to $1$ at every point of $I$. 

In this paper, using this localization,
we will interpolate a family of Teichm\"uller spaces
$T_{\sharp}^X$ between $T$ and $T_0$, where $X$ moves on any subsets of $\mathbb{S}$. Especially, we can define $T_{\sharp}^X$
as the set of all normalized piecewise symmetric homeomorphisms for a finite subset $X$ of $\mathbb{S}$. Here, by a {\it piecewise symmetric homeomorphism} $h$ for $X$, we mean that $h$ is symmetric on each closed interval of $\mathbb{S}\setminus X$. 
%Equivalently, $h$ has local dilatation equal to $1$ at every point of $\mathbb{S}\setminus X$. 
In this sense, $T_{\sharp}^X$ is a natural generalization of $T_0$, and 
the increasing scale of sets $T_0 \subset T_{\sharp}^X \subset T$ are obtained. 
Besides the case of a finite subset $X$ of $\mathbb{S}$ above, we can extend the definition of $T_{\sharp}^X$ 
to any subset $X \subset \mathbb{S}$.
This paper serves as a foundation of the theory of $T_{\sharp}^X$ and
many parts deal with properties of $T_{\sharp}^X$ which are analogous to the known properties of $T_0$. 

In Section 2, we review the standard theory of the (little) universal Teichm\"uller space. In Section 3, we define
our Teichm\"uller space $T_{\sharp}^X$ in general by using quasiconformal extension of $\rm QS$. Especially, 
we introduce piecewise symmetric homeomorphisms for a finite subset $X$ in $\mathbb{S}$ and call $T_{\sharp}^X$ 
the {\it piecewise symmetric Teichm\"uller space}. This is done by giving the intrinsic characterization of piecewise symmetric homeomorphisms for $X$ as mapping on $\mathbb{S}$ without using quasiconformal extension. 
%and compare the close relationship between the piecewise symmetric homeomorphism and the symmetric homeomorphism. 

In Sections 4--6, we show that the barycentric extension is a desired extension for piecewise symmetric homeomorphisms for $X$,
following the work of \cite{EMS}. After this, by the standard arguments,
we endow $T_{\sharp}^X$ with a complex Banach manifold structure via the Bers embedding under which it can be biholomorphically embedded as a bounded domain in a certain Banach space. As an application, in Section 7, we prove that the Bers embedding is compatible with the coset decomposition $T_{\sharp}^X\backslash T$ and the quotient of the Banach spaces. Here, we should pay attention to the definition of 
the equivalence relation given by $T_{\sharp}^X$ because it does not have a group structure unlike the usual cases.
As a consequence, we successfully endow $T_{\sharp}^X\backslash T$ with a complex structure modeled on the quotient Banach space
as in the case of the asymptotic Teichm\"uller space.

In Section 8, we show certain rigidity under conjugation by the piecewise symmetric homeomorphism for $X$.
Finally, in Section 9,
we will explore the relationship between $T_{\sharp}^X$ and $T$ when $X$ is dense in $\mathbb{S}$, and
prove that $T_\sharp^X$ is strictly included in $T$ even in this case.

Finally, we specify the difference between the Teichm\"uller spaces $T_{\sharp}^X$ in this paper and $T_{\ast}^X$ in our previous paper
\cite{WM}. For $T_{\ast}^X$, we assume that the Beltrami coefficients decay towards the boundary ${\mathbb S} \setminus X$
in a certain uniform way, but for $T_{\sharp}^X$, the decay condition is less restrictive. Due to this relaxation,
we are able to formulate an intrinsic characterization of piecewise symmetric homeomorphisms (Theorem \ref{closed}), and
a full list of the properties of the barycentric extension (Theorem \ref{EMS}). The latter result also contributes to
the property of the quotient Bers embedding (Theorem \ref{quotientBers}).

 \section{Preliminaries}
In this section, we review basic facts on the universal Teich\-m\"ul\-ler space $T$ and its little subspace $T_0$.
For details, we can refer to monographs \cite{GL, Le, Ma, Na}. 
 
Let 
$$
M(\mathbb D)=\{\mu \in L^\infty(\mathbb D) \mid \Vert \mu \Vert_\infty<1\}
$$
denote the open unit ball of the Banach space $L^{\infty}(\mathbb{D})$ 
of essentially bounded measurable functions on the unit disk $\mathbb{D}$. 
Let $M_0(\mathbb{D})$ consist of all $\mu \in M(\mathbb{D})$ vanishing at the boundary, that is, $\mu$ satisfies
$$
{\rm ess}\!\!\!\!\!\sup_{|z| \geqslant 1-t\quad}\!\!\!\!\!|\mu(z)| \to 0 \quad (t \to 0).
$$

For $\mu \in M(\mathbb{D})$, the solution of the
Beltrami equation (the measurable Riemann mapping theorem (see \cite{Ah66})) gives 
a quasiconformal homeomorphism $f$ of $\mathbb{D}$ onto itself that has complex dilatation $\mu$.
This is uniquely determined up to post-composition of an element in the group $\mbox{\rm M\"ob}(\mathbb{D})$
of M\"obius transformations of $\mathbb D$.
The quasiconformal homeomorphism $f$ extends to $\mathbb{S}$ continuously as a quasi\-symmetric homeo\-morphism of $\mathbb S$.
Conversely, any quasi\-symmetric homeo\-morphism of $\mathbb S$ extends continuously 
to a quasiconformal homeo\-morphism of $\mathbb D$.
Under a normalization condition such as $f$ keeps the points $1, i, -1$ fixed, 
$f$ is determined uniquely by $\mu \in M(\mathbb{D})$. We denote this normalized quasiconformal homeomorphism of $\mathbb D$
as well as its extension to $\mathbb S$ by $f^\mu$.
By giving the normalization, $M(\mathbb{D})$ becomes a group with operation $*$,
where $\mu \ast \nu$ for $\mu, \nu \in M(\mathbb{D})$ is defined as the complex dilatation of $f^\mu \circ f^\nu$.
The inverse $\nu^{-1}$ denotes the complex dilatation of $(f^\nu)^{-1}$.

We say that $\mu$ and $\nu$ in $M(\mathbb{D})$ are equivalent ($\mu \sim \nu$), 
if $f^{\mu}=f^{\nu}$ on the unit circle $\mathbb{S}$. We denote the equivalence class of $\mu$ by $[\mu]$. Then,
the correspondence $[\mu] \mapsto f^{\mu}|_{\mathbb{S}}$ establishes a bijection 
from $M(\mathbb{D})/{\sim}$ onto $T=\mbox{\rm M\"ob}(\mathbb{S}) \backslash \rm QS$. 
Thus, the universal Teich\-m\"ul\-ler space $T$
is identified with $M(\mathbb{D})/{\sim}$. 
The topology of $T=\mbox{\rm M\"ob}(\mathbb{S}) \backslash \rm QS$ 
coincides with the quotient topology of $M(\mathbb{D})$ induced by the {\it Teich\-m\"ul\-ler projection} $\pi:M(\mathbb D) \to T$.
The group structure on $M(\mathbb D)$ projects down to $T$. For any $[\mu], [\nu] \in T$,
$[\mu] \ast [\nu]$ is well-defined by $[\mu \ast \nu]$. Under this operation, $T$ becomes a group.

Let $B(\mathbb{D}^{*})$ denote the Banach space of functions $\varphi$ holomorphic in the exterior of the unit disk $\mathbb{D}^* = \{z\mid |z|>1\}$ with norm 
$$
\Vert \varphi \Vert_B={\rm sup}_{z \in \mathbb D^*} \rho_{\mathbb{D}^*}^{-2}(z)|\varphi(z)|.
$$
Let $B_0(\mathbb{D}^*)$ be the subspace of $B(\mathbb{D}^{*})$ consisting of all functions $\varphi$ vanishing at the boundary $\mathbb{S}$. It means that 
$$
\rho_{\mathbb{D}^*}^{-2}(z)|\varphi(z)|\to 0
$$
as $|z| \to 1^+$. Here, $\rho_{\mathbb{D}^*}(z)=(|z|^2-1)^{-1}$ denotes the hyperbolic density on $\mathbb{D}^*$.

We define a map $\Phi:M(\mathbb D) \to B(\mathbb{D}^{*})$ that sends $\mu$ to the Schwarzian derivative ${\mathcal S}(f_\mu|_{\mathbb D^*})$ of 
$f_\mu|_{\mathbb D^*}$. Here, $f_{\mu}$ is a quasiconformal homeomorphism of the complex plane $\widehat{\mathbb{C}}$ that has complex dilatation $\mu$ in $\mathbb{D}$ and is conformal in $\mathbb{D}^*$. 
The map $\Phi$ is called the {\it Bers Schwarzian derivative map}.
It is known that $\Phi$ is a holomorphic split submersion onto its image, which descends down to a homeomorphism 
$\beta: T \to B(\mathbb{D}^{*})$ onto its image, which is known as the {\it Bers embedding}. 
Via the Bers embedding, $T$ carries a natural complex structure 
so that the Teichm\"uller projection $\pi: M(\mathbb{D})\to T$ is a holomorphic split submersion. 

Similarly, the {\it little universal Teichm\"uller space} $T_0 = \mbox{\rm M\"ob}(\mathbb{S}) \backslash {\rm Sym}$ can be also defined by $\pi(M_0(\mathbb{D}))$. It is known that $\Phi(M_0(\mathbb{D})) = \Phi(M(\mathbb{D}))\cap B_0(\mathbb{D}^*)$. 
The little subspace $T_0$ is a subgroup of $T$. The quotient $T_0 \backslash T$ is defined as the {\it asymptotic Teichm\"uller space} $AT$.
The Bers embedding $\beta:T \to B(\mathbb{D}^*)$ is compatible with the coset decomposition $T_0\setminus T$ and the quotient Banach space $B_0(\mathbb{D}^*)\backslash B(\mathbb{D}^*)$. In fact, the {\it quotient Bers embedding}
$\hat \beta: T_0\setminus T \to B_0(\mathbb{D}^*)\backslash B(\mathbb{D}^*)$ is well-defined to be a homeomorphism onto the image.
By which the complex structure modeled on the quotient Banach space $B_0(\mathbb{D}^*)\backslash B(\mathbb{D}^*)$ is provided for $AT$. 
These facts were proved in \cite{GS} and \cite{EMS}.

The {\it barycentric extension} due to Douady and Earle \cite{DE} gives a quasiconformal extension $E(h):\mathbb{D} \to \mathbb{D}$ of
any quasisymmetric homeomorphism $h \in {\rm QS}$ in a conformally natural way.
In fact, the quasiconformal extension $E(h)$ is a diffeomorphism of $\mathbb{D}$ that is bi-Lipschitz with respect to the hyperbolic metric.
The conformal naturality means that $E(g_1 \circ h \circ g_2)=E(g_1) \circ E(h) \circ E(g_2)$ is satisfied for any $h \in {\rm QS}$ and 
any $g_1, g_2 \in \mbox{\rm M\"ob}(\mathbb{S})$, where the extensions $E(g_1)$ and $E(g_2)$ are in $\mbox{\rm M\"ob}(\mathbb{D})$.
The barycentric extension
induces a continuous (in fact, real analytic) section $s: T \to M(\mathbb{D})$ of the Teich\-m\"ul\-ler projection $\pi:M(\mathbb{D}) \to T$ ($\pi \circ s={\rm id}_T$) by sending a point $[\mu] \in T$ to the complex dilatation $s([\mu]) \in M(\mathbb{D})$ of $E(f^{\mu}|_{\mathbb{S}})$. 
It was proved in \cite{EMS} that $s$ maps $T_0$ into $M_0(\mathbb D)$.

\section{Piecewise symmetric Teichm\"uller space}

In this section, we introduce a new family $T_{\sharp}^X$ of Teichm\"uller spaces which gives an interpolation between $T$ and $T_0$. 

Let $X$ be any subset of $\mathbb{S}$.
We say that $\mu \in L^\infty(\mathbb D)$ {\it vanishes at the boundary relative to $X$} if for every $\varepsilon >0$,
there exists a compact subset $K$ of ${\mathbb D} \cup X$ such that 
$$\Vert \mu|_{{\mathbb D}\setminus K} \Vert_\infty<\varepsilon.$$
Let $L_{\sharp}^{X}(\mathbb{D})$ denote the set of all $\mu \in L^\infty(\mathbb D)$ that vanish at the boundary relative to $X$,
and let $M_{\sharp}^{X}(\mathbb{D})=L_{\sharp}^{X}(\mathbb{D})\cap M(\mathbb{D})$. 
We see that $L_{\sharp}^X(\mathbb{D})$ is a closed subspace of $L^{\infty}(\mathbb{D})$. Indeed, 
assuming that a sequence $\{\mu_k\}_{k \in \mathbb{N}}$ in $L_{\sharp}^X(\mathbb{D})$ and $\mu \in L^{\infty}(\mathbb{D})$ are 
given so that $\Vert\mu_k-\mu\Vert_{\infty}\to 0$ as $k\to\infty$, we show that $\mu\in L_{\sharp}^X(\mathbb{D})$.
For each $\varepsilon>0$, we can choose some $k_0 \in\mathbb{N}$ 
such that $\Vert\mu_{k_0}-\mu\Vert_{\infty}< \varepsilon$. Since $\mu_{k_0}\in L_{\sharp}^X(\mathbb{D})$, 
there exists some compact subset $K$ of ${\mathbb D} \cup X$ such that $\Vert \mu_{k_0}|_{\mathbb{D} \setminus K}\Vert_{\infty}<\varepsilon$. Thus, 
$$
\Vert\mu|_{\mathbb{D} \setminus K}\Vert_{\infty} \leqslant \Vert\mu_{k_0}|_{\mathbb{D} \setminus K}\Vert_{\infty}+\Vert\mu_{k_0}-\mu\Vert_{\infty}<2\varepsilon,
$$
which implies that $\mu\in L_{\sharp}^X(\mathbb{D})$.

\begin{definition}
{\rm
For $X \subset \mathbb{S}$, we denote by ${\rm QS}_{\sharp}^X$ the subset of $\rm QS$ consisting of all
quasisymmetric homeomorphisms
obtained by the boundary extension of quasiconformal homeo\-morphisms of $\mathbb{D}$ onto itself 
(not necessarily normalized) with
dilatations $\mu$ in $M_{\sharp}^{X}(\mathbb{D})$.
The {\it symmetric Teich\-m\"ul\-ler space $T_{\sharp}^X$
relative to $X$} is defined as 
$$
T_{\sharp}^X=\mbox{\rm M\"ob}(\mathbb{S}) \backslash {\rm QS}_{\sharp}^X=\pi(M_{\sharp}^X(\mathbb{D})).
$$
}
\end{definition}

By the composition or the inverse of quasisymmetric homeomorphisms, 
the subset $X$ may be mapped to another subset $Y$ of $\mathbb S$. 
On account of this, ${\rm QS}_{\sharp}^X$ is not a subgroup of $\rm QS$, and similarly,
$T_{\sharp}^X$ is not a subgroup of $T$ unless either $X$ or ${\mathbb S} \setminus X$ consists of less than or equal to three points.

In the remainder of this section, we focus on the case where $X \subset \mathbb S$ consists of
finitely many points. In this case, we call an element of ${\rm QS}_{\sharp}^X$
a {\it piecewise symmetric homeomorphism} for $X$ and $T_{\sharp}^X$
the {\it piecewise symmetric Teich\-m\"ul\-ler space} for $X$.

For a finite subset $X \subset \mathbb S$, 
we easily see the decomposition of the Banach space of the Beltrami differentials as follows.
 
\begin{proposition}\label{L}
For $X=\{\xi_1,\dots \xi_n\} \subset \mathbb {S}$,
$L_{\sharp}^{X}(\mathbb{D})=L_{\sharp}^{\xi_1}(\mathbb{D})+\dots +L_{\sharp}^{\xi_n}(\mathbb{D})$.
\end{proposition}

\begin{proof}
The inclusion $\supset$ is easy to see. 
For the inverse inclusion $\subset$, we take any element $\mu$ in $L_{\sharp}^{X}(\mathbb{D})$. 
The unit circle $\mathbb{S}$ is divided into $n$ sub-arcs by the points $\xi_1, \dots, \xi_n$. 
Take the midpoint of each sub-arc and connect the midpoint of each sub-arc to the origin $0$ by a segment. 
The union of these segments divide $\mathbb{D}$ into $n$ sectors $E_1, \dots, E_n$, and
each $E_i$ $(i=1,\dots, n)$ contains only one $\xi_i$ on its boundary. 
Then, the decomposition of $\mu$ is given simply by restricting $\mu$ to each sector;
$\mu=\mu 1_{E_1}+ \dots +\mu 1_{E_n}$, where $\mu 1_{E_i} \in L_{\sharp}^{\xi_i}(\mathbb{D})$ for each $i=1,\dots,n$.
\end{proof}

We consider the intrinsic characterization of piecewise symmetric homeomorphisms $h \in {\rm QS}_{\sharp}^X$ 
as mapping on $\mathbb{S}$. Before stating our result, we recall some terminology.

The {\it local dilatation} $D_h(\eta)$ of a quasisymmetric homeomorphism $h$ {\it at a point} $\eta \in \mathbb{S}$
is the infimum of the dilatations of the possible quasiconformal extensions $\tilde{h}$ 
of $h$ to open sets $V$ of $\overline{\mathbb D}$ with $\eta \in V$.  
This concept can be extended obviously to the case where $\eta$ is replaced by a closed interval $I$ of $\mathbb{S}$. The {\it local dilatation} $D_h(I)$ of a quasisymmetric homeomorphism $h$ {\it on a closed interval} $I \subset \mathbb{S}$ 
is the infimum of the dilatations of the possible quasiconformal extensions $\tilde{h}$ of $h$ to open sets $V$ of $\overline{\mathbb D}$ with $I \subset V$.

The following result is known in the theory of quasiconformal mapping (see \cite[Staz 3.1]{Fe} and \cite[Proposition 3.1]{GS} for more details).  

\begin{proposition}\label{sym}
For a quasisymmetric homeo\-morphism $h$ from a closed interval $I$ of $\mathbb{S}$ to a closed interval $J$ of $\mathbb{S}$, the following conditions on $h$ are equivalent:
\begin{enumerate}
    \item $h$ is symmetric on $I$;
    \item $h$ has the local dilatation equal to 1 at every point of $I$;
    \item $h$ has the local dilatation equal to 1 on $I$;
    \item there exists an extension $\tilde{h}$ of $h$ to an open subset $V$ of $\overline{\mathbb D}$
    with $I \subset V$ that is asymptotically conformal on $V \cap \mathbb D$.
\end{enumerate}
\end{proposition}

\begin{remark}\label{ac}
To clarify condition $(4)$, by saying that the extension $\tilde{h}$ of $h$ on $I$ is asymptotically conformal on $V \cap \mathbb D$, 
we mean that, for every $\varepsilon > 0$, there is an open subset $U$ of $\overline{\mathbb D}$ with $I \subset U \subset V$ 
such that $|\mu(z)| < \varepsilon$ for almost all $z$ in $U \cap \mathbb D$. 
\end{remark}

Now we can characterize a piecewise symmetric homeomorphism for a finite set $X \subset \mathbb S$ without using quasiconformal extension.

\begin{theorem}\label{closed}
For a quasisymmetric self-homeomorphism $h$ of $\mathbb{S}$, $h \in {\rm QS}_{\sharp}^{X}$ if and only if $h|_{I}$ is symmetric for each closed interval $I$ contained in $\mathbb{S}\setminus X$. 
\end{theorem}

\begin{proof}
Suppose that $h \in {\rm QS}_{\sharp}^{X}$. Then, there is an extension $\tilde{h}$ of $h$ to $\mathbb{D}$ with complex dilatation $\mu \in M_{\sharp}^{X}(\mathbb{D})$, which implies that for every $\varepsilon > 0$, there exists a compact subset $K$ of $\mathbb{D}\cup X$ such that $\Vert \mu|_{\mathbb{D}\backslash K}\Vert_{\infty} < \varepsilon$. We conclude by $(4) \Rightarrow (1)$ in Proposition \ref{sym} that $h$ is symmetric on each closed interval $I \subset \mathbb{S}\backslash X$.

Conversely, suppose that $h|_I$ is symmetric for each closed interval $I$ contained in $\mathbb{S}\backslash X$. For any $\eta \in \mathbb{S}\backslash X$, let $I_{\eta}$ be an open interval on $\mathbb{S}\backslash X$ with $X\cap \overline{I_{\eta}}=\varnothing$, containing $\eta$. %and symmetric with respect to $\eta$. 
Hu and Muzician \cite[Theorem 2]{HM} showed that there exists an open subset $V_\eta$ of $\overline{\mathbb D}$
%(in fact, $V_\eta \cap \mathbb D$ is a hyperbolic half-plane)
with $V_{\eta}\cap \mathbb{S} \subset I_{\eta}$ such that the barycentric extension $E(h):\mathbb D \to \mathbb D$ of $h \in \rm QS$ is asymptotically conformal on $V_\eta \cap \mathbb D$. 
For any $\varepsilon>0$, we take an open subset $U_\eta$ of $\overline{\mathbb D}$ with $\eta \in U_\eta \subset V_\eta$
such that the complex dilatation $\mu$ of $E(h)$ satisfies that $|\mu(z)| < \varepsilon$ for almost all $z$ in $U_\eta \cap \mathbb D$ (see Remark \ref{ac} above). We
set $U=\bigcup_{\eta\in\mathbb{S}\backslash X}(U_{\eta} \cap \mathbb D)$ and $K =(\mathbb{D} \cup X) \setminus U$.
Then, $K$ is a compact subset of $\mathbb{D}\cup X$ and $\Vert \mu|_{{\mathbb D} \setminus K}\Vert_\infty \leqslant \varepsilon$. This
shows that $\mu \in M_{\sharp}^{X}(\mathbb{D})$, and consequently, $h \in {\rm QS}_{\sharp}^{X}$. 
\end{proof}

\begin{remark}
{\rm
Following Proposition \ref{sym} and Theorem \ref{closed}, we find out that for a quasisymmetric self-homeomorphism $h$ of $\mathbb{S}$, $h$ is symmetric if and only if $h$ has a local dilatation equal to 1 at every point of $\mathbb{S}$, while $h$ is piecewise symmetric for $X$ if and only if $h$ has a local dilatation equal to 1 at every point of $\mathbb{S}\backslash X$. In this sense, the piecewise symmetric homeomorphism  for $X$ is a natural generalization of the symmetric homeomorphism on $\mathbb{S}$. 
}
\end{remark}

\section{Bers Schwarzian derivative map}

In this section, we focus on the Bers Schwarzian derivative map $\Phi:M(\mathbb{D}) \to B(\mathbb{D}^*)$
restricted to the subspace $M_{\sharp}^X(\mathbb{D})$. 

We first introduce the corresponding subspace of $B(\mathbb{D}^*)$.
We say that $\varphi \in B(\mathbb{D}^*)$ {\it vanishes at the boundary relative to $X \subset \mathbb S$} if for every $\varepsilon >0$,
there exists a compact subset $K^*$ of ${\mathbb D^*} \cup X$ such that 
$\Vert \rho_{\mathbb D^*}^{-2}(z)\varphi(z)|_{{\mathbb D}^*\setminus K^*} \Vert_\infty<\varepsilon$.
Let $B_{\sharp}^{X}(\mathbb{D}^*)$ denote the set of all $\varphi \in B(\mathbb D^*)$ that vanish at the boundary relative to $X$.
We see that $B_{\sharp}^{X}(\mathbb{D}^*)$ is a closed subspace of $B(\mathbb D^*)$ by a similar proof to the case of  
$L_{\sharp}^X(\mathbb{D}) \subset L^{\infty}(\mathbb{D})$. 

By the following theorem, we see that $B_{\sharp}^X(\mathbb{D}^*)$ is the appropriate space 
corresponding to $M_{\sharp}^X(\mathbb{D})$
under the Bers Schwarzian derivative map $\Phi$.

\begin{theorem}\label{Bers}
%For any $X \subset \mathbb{S}$,
The Bers Schwarzian derivative map $\Phi$ maps $M_{\sharp}^X(\mathbb{D})$ into $B_{\sharp}^X(\mathbb{D}^*)$. 
\end{theorem}

\begin{proof}
By the integral representation of the Schwarzian derivative, which was established by Astala and Zinsmeister \cite{AZ}
(see also Cui \cite{Cu}), we have  
$$
\rho_{\mathbb{D}^*}^{-4}(\zeta^*)|\Phi(\mu)(\zeta^*)|^2 \leqslant C\int_{\mathbb{D}}\frac{(|\zeta^*|^2-1)^2}
{|z-\zeta^*|^4}|\mu(z)|^2dxdy 
$$
for every $\zeta^* \in \mathbb{D}^*$, 
where $C>0$ is a constant depending only on 
$\Vert\mu\Vert_{\infty}.$ 

Let $\gamma_{\zeta}(z)=(\overline{\zeta^*}z-1)/(z-\zeta^*) \in \mbox{\rm M\"ob}(\mathbb{D})$ 
be a M\"obius transformation of $\mathbb{D}$ onto itself that sends $\zeta$ to $0$. 
Here, $\zeta \in \mathbb{D}$ and $\zeta^* \in \mathbb{D}^*$ are the reflection to each other with respect to $\mathbb{S}$. 
We see that $|\gamma_{\zeta}'(z)|^2=(|\zeta^*|^2-1)^2/|z-\zeta^*|^4$. It follows that 
\begin{equation*}
 \begin{split}
 &\quad \int_{\mathbb{D}}\frac{(|\zeta^*|^2-1)^2}
{|z-\zeta^*|^4}|\mu(z)|^2dxdy  =   \int_{\mathbb{D}}|\gamma_{\zeta}'(z)|^2|\mu(z)|^2dxdy\\
&=\int_{\mathbb{D} \setminus K}|\gamma_{\zeta}'(z)|^2|\mu(z)|^2dxdy + \int_{K}|\gamma_{\zeta}'(z)|^2|\mu(z)|^2dxdy.\\
 \end{split}   
\end{equation*}
Here, for a given $\varepsilon > 0$, we choose a compact subset $K$ of $\mathbb{D}\cup X$
so that $\Vert \mu|_{\mathbb{D} \setminus K} \Vert_{\infty} < \varepsilon$
under the condition $\mu \in M_\sharp^X(\mathbb{D})$.
Then, the last formula is estimated from above by 
\begin{equation*}
 \begin{split}
 &\quad\ \varepsilon^2 \int_{\mathbb{D} \setminus K}|\gamma_{\zeta}'(z)|^2dxdy + \int_{K}|\gamma_{\zeta}'(z)|^2dxdy \\
 &\leqslant\pi\varepsilon^2 + {\rm Area}(\gamma_{\zeta}(K)),\\
 \end{split}   
\end{equation*}
where $\rm Area$ stands for the Euclidean area. 

We consider ${\rm Area}(\gamma_{\zeta}(K))$ for $\zeta \in {\mathbb D} \setminus K$.
The notation $\asymp$ is used below when the both sides are 
comparable, i.e., one side is bounded from above and below by multiples of the other side
with some positive absolute constants. The notation $\lesssim$ is used when the left side is bounded from above by a multiple of the right side with some positive absolute constant. 

Noting that ${\rm Area}(\gamma_{\zeta}(K))\lesssim 1 - d(0, \gamma_{\zeta}(K))$ for the Euclidean distance $d$, we see that 
\begin{equation*}
 1 - d(0, \gamma_{\zeta}(K))
 \asymp e^{-d_H(0, \gamma_{\zeta}(K))} = e^{-d_H(\zeta, K)} = e^{-d_H(\zeta^*, K^*)}
\end{equation*}
by the hyperbolic distance formula $d_H(0,z)=\log\frac{1+|z|}{1-|z|}$ $(z \in \mathbb{D})$ and its conformal invariance.
Therefore, a condition $d_H(\zeta^*, K^*) > -\log \varepsilon$ implies that ${\rm Area}(\gamma_{\zeta}(K)) < A\varepsilon$ for some absolute constant $A>0$. We set
$$
E^*=\{\zeta^* \in {\mathbb D}^* \mid d_H(\zeta^*,K^*) \leqslant -\log \varepsilon \},
$$
which is a compact subset of $\mathbb{D}^*\cup X$. If $\zeta^* \in {\mathbb D}^* \setminus E^*$, then $d_H(\zeta^*, K^*) > -\log \varepsilon$.

Combining this area estimate with the above integral inequality, we conclude  that if $\zeta^* \in \mathbb{D}^*\setminus E^*$, then 
$$
\rho_{\mathbb{D}^*}^{-2}(\zeta^*)|\Phi(\mu)(\zeta^*)| < \sqrt{\pi\varepsilon^2+A\varepsilon}.
$$
Since $\varepsilon > 0$ is arbitrarily chosen, this implies that 
%if $\mu \in M_\sharp^X(\mathbb{D})$, then
$\Phi(\mu) \in B_\sharp^X(\mathbb{D}^*)$.
\end{proof}

We note that $\Phi:M_{\sharp}^X(\mathbb{D}) \to B_{\sharp}^X(\mathbb{D}^*)$ is holomorphic because
$\Phi:M(\mathbb{D}) \to B(\mathbb{D}^*)$ is holomorphic and the closed subspaces 
$M_{\sharp}^X(\mathbb{D})$ and $B_{\sharp}^X(\mathbb{D}^*)$ are endowed with
the relative topologies from $M(\mathbb{D})$ and $B(\mathbb{D}^*)$.
See Theorem \ref{submersion} below.
 
\section{Barycentric extension}

In this section, we will prove that the barycentric extension 
gives an appropriate right inverse of $\pi:M_{\sharp}^{X}(\mathbb{D}) \to T_{\sharp}^X$.
In other words,
for the section $s:T \to M(\mathbb{D})$ of the universal Teich\-m\"ul\-ler space induced by the barycentric extension,
we show that the image $s(T_{\sharp}^X)$ is in $M_{\sharp}^{X}(\mathbb{D})$.

%\begin{lemma}
%Let $g:\mathbb{D} \to \mathbb{D}$ be a quasiconformal self-homeomorphism of $\mathbb{D}$.
%If $\mu_g \in M_{\sharp}^X(\mathbb{D})$, then $\mu_{g^{-1}} \in M_{\sharp}^Y(\mathbb{D})$ for %$Y=g(X)$,
%where $g$ is assumed to be extended to $\mathbb S$.
%\end{lemma}

%\begin{theorem}\label{DE}
%For every $h\in \rm QS_{\sharp}^{X}$, the complex dilatation of the barycentric extension $E(h)$ %of $h$ is in $M_{\sharp}^{X}(\mathbb{D})$.
%\end{theorem}

This claim  follows from the following more general result concerning the section $s$.
This was originally proved by Earle, Markovic, and Saric \cite[Theorem 4]{EMS} for the little universal Teich\-m\"ul\-ler space 
$T_0=\mbox{\rm M\"ob}(\mathbb{S}) \backslash {\rm Sym}$ and for
the subspaces
$M_0(\mathbb{D}) \subset M(\mathbb{D})$ and $B_0(\mathbb{D}^*) \subset B(\mathbb{D}^*)$
consisting of vanishing elements on the boundary.
The proof below is a modification of theirs.
%For the sake of simplicity, the normalization for a quasiconformal homeomorphism $f^\mu$ of $\mathbb{D}$ onto itself with
%$\mu \in M(\mathbb{D})$ only in this proof is given by imposing $f^\mu(0)=0$ and $f^\mu(1)=1$ on $f^\mu$.

\begin{theorem}\label{EMS}
Let $\mu$ and $\nu$ be in $M(\mathbb{D})$, and let $X \subset \mathbb{S}$.
Then, the following are equivalent:
\begin{enumerate}
\item
$\Phi(\mu)-\Phi(\nu)\in B_{\sharp}^X(\mathbb{D}^*)$;
\item
$s([\mu]) - s([\nu])\in L_{\sharp}^X(\mathbb{D})$;
\item
$s([\mu]) \ast s([\nu])^{-1} \in M_{\sharp}^Y(\mathbb{D})$ for $Y=f^\nu(X) \subset \mathbb{S}$;
\item
$[\mu] \ast [\nu]^{-1} \in T_{\sharp}^Y$ for $Y=f^\nu(X) \subset \mathbb{S}$.
\end{enumerate}
\end{theorem} 

\begin{proof}
$(1) \Rightarrow (2)$: For any point $\eta \in {\mathbb S} \setminus X$, 
we take a sequence $\{z_k\}_{k \in \mathbb {N}} \subset \mathbb{D}$ that converges to $\eta$. 
For each $k$, we choose a M\"obius transformation $g_k \in \mbox{\rm M\"ob}(\mathbb{D})$ with $g_k(0)=z_k$,
and define $\mu_k=g_k^* s([\mu])$ and $\nu_k=g_k^* s([\nu])$. Then, $\Phi(\mu_k)=g_k^* \Phi(\mu)$ and 
$\Phi(\nu_k)=g_k^* \Phi(\nu)$ for $g_k \in \mbox{\rm M\"ob}(\mathbb{D}^*)$. 
We also see that $\{g_k(z^*)\}$ converges to $\eta$ for every $z^* \in \mathbb {D}^*$. 
Since we assume that $\Phi(\mu)-\Phi(\nu) \in B_{\sharp}^X(\mathbb{D}^*)$, we see that
$$
\rho_{\mathbb{D}^*}^{-2}(z^*)|\Phi(\mu_k)(z^*)-\Phi(\nu_k)(z^*)|
=\rho_{\mathbb{D}^*}^{-2}(g_k(z^*))|(\Phi(\mu)-\Phi(\nu))(g_k(z^*))| 
$$
tends to $0$ as $k \to \infty$ for each $z^* \in \mathbb{D}^*$.
In particular, $\Phi(\mu_k)-\Phi(\nu_k) \to 0$ as $k \to \infty$.

Since $\Vert \mu_k \Vert_\infty=\Vert s([\mu]) \Vert_\infty$ and $\Vert \nu_k \Vert_\infty=\Vert s([\nu]) \Vert_\infty$,
by passing to a subsequence, we may assume that
$f^{\mu_{k}}$ converges uniformly to some 
quasiconformal homeomorphism $f^{\mu_0}$ with a complex dilatation $\mu_0 \in M({\mathbb D})$ and $f^{\nu_{k}}$ 
converges uniformly to some 
$f^{\nu_0}$ with $\nu_0 \in M({\mathbb D})$. In this situation,
\cite[Lemma 6.1]{EMS} asserts that 
$\Phi(\mu_{k})$ converges locally uniformly to
$\Phi(\mu_0)$ and $\Phi(\nu_{k})$ converges locally uniformly to
$\Phi(\nu_0)$ on $\mathbb{D}^*$.  
Since $\Phi(\mu_k)-\Phi(\nu_k) \to 0$ as $k \to \infty$, this implies that $\Phi(\mu_0)=\Phi(\nu_0)$.

By \cite[Lemma 6.1]{EMS} again, we see that $s([\mu_{k}])$ converges locally uniformly to
$s([\mu_0])$ and $s([\nu_{k}])$ converges locally uniformly to
$s([\nu_0])$ on $\mathbb{D}$. Here, $\Phi(\mu_0)=\Phi(\nu_0)$ implies that $s([\mu_0])=s([\nu_0])$.
Therefore, $s([\mu_{k}])-s([\nu_{k}])$ converges to $0$, and in particular,
$s([\mu_{k}])(0)-s([\nu_{k}])(0) \to 0$ as $k \to \infty$.

The conformal naturality of the barycentric extension implies that
$$
s([\mu_{k}])=s([g_k^*\mu])=g_k^*(s([\mu])); \quad s([\nu_{k}])=s([g_k^*\nu])=g_k^*(s([\nu])).
$$
It follows that
$$
|s([\mu])(z_k)-s([\nu])(z_k)|=|s([\mu_k])(0)-s([\nu_k])(0)| \to 0 \quad (k \to \infty).
$$
Since $s([\mu])-s([\nu])$ is continuous and $\{z_k\}$ is an arbitrary sequence converging to a point on ${\mathbb S} \setminus X$,
this implies that $s([\mu])-s([\nu]) \in L_{\sharp}^X(\mathbb{D})$.

$(2) \Rightarrow (3)$: Let $\lambda=s([\mu]) \ast s([\nu])^{-1}$, that is,
$\lambda$ is the complex dilatation of the composition $f^{s([\mu])} \circ (f^{s([\nu])})^{-1}$. This satisfies
$$
|\lambda \circ f^{s([\nu])}|=\frac{|s([\mu]) -s([\nu])|}{|1-\overline{s([\nu])} s([\mu])|},
$$
from which the assertion follows.

$(3) \Rightarrow (4)$:
From $\pi(s([\mu]) \ast s([\nu])^{-1})=[\mu] \ast [\nu]^{-1}$, the assertion follows immediately.

$(4) \Rightarrow (1)$:  
There are $\mu' \in [\mu]$ and $\nu' \in [\nu]$ such that $\lambda=\mu' \ast \nu'^{-1} \in M_{\sharp}^Y(\mathbb D)$.
As before, for a point $\eta \in {\mathbb S} \setminus X$ and a sequence $\{z_k\}_{k \in \mathbb {N}} \subset \mathbb{D}$ converging to $\eta$,
we choose $g_k \in \mbox{\rm M\"ob}(\mathbb{D})$ with $g_k(0)=z_k$,
and define $\mu_k=g_k^* \mu'$ and $\nu_k=g_k^* \nu'$. By
$$
|\lambda \circ f^{\nu'}|=\frac{|\mu' -\nu'|}{|1-\overline{\nu'} \mu'|} \in M_{\sharp}^X(\mathbb D),
$$
we see that
$$
\Vert (\mu_k-\nu_k)|_{\Delta(0,r)} \Vert_\infty=\Vert (\mu'- \nu')|_{\Delta(z_k,r)} \Vert_\infty
$$
tends to $0$ as $k \to \infty$ for any $r>0$. Here, $\Delta(z,r) \subset \mathbb{D}$ denotes a hyperbolic disk
with center $z$ and radius $r$.

Since $\Vert \mu_k \Vert_\infty=\Vert \mu' \Vert_\infty$ and $\Vert \nu_k \Vert_\infty=\Vert \nu' \Vert_\infty$,
by passing to a subsequence, we may assume that  
$f^{\mu_{k}}$ converges uniformly to some 
quasiconformal homeomorphism $f^{\mu_0}$ with a complex dilatation $\mu_0 \in M({\mathbb D})$ and $f^{\nu_{k}}$ 
converges uniformly to some 
$f^{\nu_0}$ with $\nu_0 \in M({\mathbb D})$. Let $\lambda_k=\mu_k \ast \nu_k^{-1}$.
%, that is,
%$\lambda_k$ is the complex dilatation of $f^{\mu_k} \circ (f^{\nu_k})^{-1}$. This satisfies
%$$
%|\lambda_k \circ f^{\nu_k}|=\frac{|\mu_k -\nu_k|}{|1-\overline{\nu_k} \mu_k|}.
%$$

For an arbitrary compact subset $E \subset \mathbb{D}$, we take $r>0$ such that $(f^{\nu_0})^{-1}(E) \subset \Delta(0,r)$.
Since $(f^{\nu_{k}})^{-1}$ converges to $(f^{\nu_0})^{-1}$ uniformly on $\mathbb{D}$ as $k \to \infty$, we can assume that
$(f^{\nu_{k}})^{-1}(E) \subset \Delta(0,r)$ for all sufficiently large $k$. Hence, $\Vert \lambda_{k}|_E \Vert_\infty \to 0$
as $k \to \infty$.
%$$
%\Vert \lambda_{k}|_E \Vert_\infty \leq 
%\frac{\Vert (\mu_{k}-\nu_{k})|_{\Delta(0,r)}\Vert_\infty}{1-\Vert \mu \Vert_\infty \Vert \nu \Vert_\infty}
%\to 0 \quad (k \to \infty).
%$$
Since $E$ is arbitrary, we see from this estimate that the limit $f^{\mu_0} \circ (f^{\nu_0})^{-1}$ of
$f^{\mu_{k}} \circ (f^{\nu_{k}})^{-1}$ is conformal on $\mathbb{D}$. In fact, $f^{\mu_0} \circ (f^{\nu_0})^{-1}$ is the identity
by the normalization. Therefore, $f^{\mu_0}= f^{\nu_0}$, and both $f^{\mu_{k}}$ and $f^{\nu_{k}}$
converge uniformly to the same limit $f^{\mu_0}$ as $k \to \infty$.

For every $\mu \in M(\mathbb D)$, we define $\widetilde \Phi(\mu)(z)=z^4\Phi(\mu)(z)$ ($z \in \mathbb D^*$).
As $\rho_{\mathbb D^*}^{-2}(z)|\Phi(\mu)(z)|$ is bounded, we see that $\widetilde \Phi(\mu)$ is a holomorphic function on $\mathbb D^*$.
Similarly to \cite[Lemma 6.1]{EMS}, it can be proved that 
$\widetilde \Phi(\mu_{k})$ and $\widetilde \Phi(\nu_{k})$ converge to the same limit $\widetilde \Phi(\mu_0)$
locally uniformly on $\mathbb{D}^*$ as $k \to \infty$. Therefore, $\widetilde \Phi(\mu_{k})-\widetilde \Phi(\nu_{k})$ 
converges to $0$, and in particular,
$\widetilde \Phi(\mu_{k})(\infty)-\widetilde \Phi(\nu_{k})(\infty) \to 0$ as $k \to \infty$.

The equivariance of the Bers projection implies that 
$$
\Phi(\mu_k)=\Phi(g_k^*\mu')=g_k^* \Phi(\mu);\quad \Phi(\nu_k)=\Phi(g_k^*\nu')=g_k^* \Phi(\nu).
$$
By $\lim_{z \to \infty}g_k(z)=z_k^*$ and $\lim_{z \to \infty}|z^2g'_k(z)|=\rho_{\mathbb{D}^*}^{-1}(z_k^*)$,
it follows that
\begin{align*}
\rho_{\mathbb{D}^*}^{-2}(z_k^*)|\Phi(\mu)(z_k^*)-\Phi(\nu)(z_k^*)|
&=\lim_{z \to \infty}|z^2g'_k(z)|^2|\Phi(\mu)(g_k(z))-\Phi(\nu)(g_k(z))|\\
&=|\widetilde\Phi(\mu_k)(\infty)-\widetilde\Phi(\nu_k)(\infty)|.
\end{align*}
This tends to $0$ as $k \to \infty$.
This implies that $\Phi(\mu)-\Phi(\nu) \in B_{\sharp}^X(\mathbb{D})$.
%$(2) \Rightarrow (3)$: 
%We take an arbitrary sequence $\{w_k\}_{k \in \mathbb {N}} \subset \mathbb{D}$ such that
%$w_k \in \mathbb{D} \setminus \bigcup_{i=1}^n D_{1/k}^{\eta_i}$ ($Y=\{\eta_i\}_{i=1}^n$) for every $k \in \mathbb {N}$.
%For each $k$, we choose a M\"obius transformation $h_k \in \mbox{\rm M\"ob}(\mathbb{D})$ with $h_k(0)=z_k$
%and define $\nu^{-1}_k=h_k^* s([\nu^{-1}])$. We also choose a M\"obius transformation $g_k \in \mbox{\rm M\"ob}(\mathbb{D})$
%such that $g_k^{-1} \circ f^{s([\nu^{-1}])} \circ h_k$ fixes $0$ and $1$.
%Let $z_k=g_k(0)$ for $k \in \mathbb{N}$. Then, for every $z^* \in \mathbb{D}^*$ and for every $\tilde k \in \mathbb{N}$,
%there exists some $k_0 \in \mathbb{N}$ such that
%$$
%g_k(z^*) \in \mathbb{D}^* \setminus \bigcup_{i=1}^n D_{1/\tilde k}^{\xi_i}
%$$ 
%for every $k \geq k_0$.
\end{proof}

Here are direct consequences from this theorem.

\begin{corollary}\label{DE}
For every $h\in \rm QS_{\sharp}^{X}$, the complex dilatation of the barycentric extension $E(h)$ is in $M_{\sharp}^{X}(\mathbb{D})$.
Hence, we have a global continuous section $s:T_{\sharp}^X \to M_{\sharp}^{X}(\mathbb{D})$ to the 
Teich\-m\"ul\-ler projection $\pi:M_{\sharp}^{X}(\mathbb{D}) \to T_{\sharp}^X$.
\end{corollary}

\begin{proof}
By setting $\nu=0$ in Theorem \ref{EMS}, we obtain that the condition
$s([\mu]) \in M_{\sharp}^{X}(\mathbb{D})$ is equivalent to that $\Phi(\mu) \in B_{\sharp}^X(\mathbb{D}^*)$.
Let $\mu \in M_{\sharp}^{X}(\mathbb{D})$ be the complex dilatation of any quasiconformal extension of $h\in \rm QS_{\sharp}^{X}$.
Then, the complex dilatation of $E(h)$ is $s([\mu])$.
Since $\Phi(\mu) \in B_{\sharp}^X(\mathbb{D}^*)$ by Theorem \ref{Bers}, we see that $s([\mu]) \in M_{\sharp}^{X}(\mathbb{D})$.
\end{proof}

\begin{remark}
{\rm
Examining the proof of Theorem \ref{closed} relying on \cite{HM}, we find out that 
for every $h\in \rm QS_{\sharp}^{X}$, the complex dilatation of the barycentric extension $E(h)$ of $h$ is in $M_{\sharp}^{X}(\mathbb{D})$
in the case that $X \subset \mathbb S$ is a finite set. 
This gives another proof of Corollary \ref{DE} in this case.
}
\end{remark}

\begin{corollary}\label{contractible} 
The Teich\-m\"ul\-ler space $T_{\sharp}^X$ is contractible.
\end{corollary}

\begin{proof}
Since $M_{\sharp}^{X}(\mathbb{D})$ is contractible, the assertion follows from Corollary \ref{DE}.
\end{proof}

\begin{corollary}\label{image}
$\beta(T_{\sharp}^X) = \beta(T)\cap B_{\sharp}^X(\mathbb{D}^*)$.
\end{corollary}
\begin{proof}
Theorem \ref{Bers} implies that $\beta(T_{\sharp}^X) \subset \beta(T)\cap B_{\sharp}^X(\mathbb{D}^*)$. 
By taking $\nu = 0$ in Theorem \ref{EMS}, 
we see that the converse inclusion is also true.
\end{proof}

\section{Holomorphic split submersion}

In this section, we will endow $T_{\sharp}^{X}$ with a complex Banach manifold structure. 
This is done by the investigations of the Bers Schwarzian derivative map $\Phi:M_{\sharp}^X(\mathbb{D})\to B_{\sharp}^X(\mathbb{D}^*)$
given in Theorem \ref{Bers}.
We note that the image of $\Phi$ is 
$\beta(T_{\sharp}^X) = \beta(T)\cap B_{\sharp}^X(\mathbb{D}^*)$ by Corollary \ref{image}, which is an open subset of $B_{\sharp}^X(\mathbb{D}^*)$.

We recall that the right translation $r_\nu$ for any $\nu \in M(\mathbb D)$ defined by
$r_\nu(\mu)=\mu \ast \nu^{-1}$ for every $\mu \in M(\mathbb D)$ is a biholomorphic automorphism of $M(\mathbb D)$.
Concerning the restriction of these automorphisms to $M_{\sharp}^X(\mathbb D)$, 
we in particular see that $r_\nu$ is a biholomorphic automorphism of $M_{\sharp}^X(\mathbb D)$
for any $\nu \in M_{\sharp}^X(\mathbb D)$ with $[\nu]=[0]$ (see \cite[Lemma 6.1]{WM}).
We also see that any equivalent Beltrami coefficients $\mu_1, \mu_2 \in M_{\sharp}^X(\mathbb D)$ 
are mapped to one another by a biholomorphic automorphism $r_\nu$ of $M_{\sharp}^X(\mathbb D)$ for 
some $\nu \in M_{\sharp}^X(\mathbb D)$ with $[\nu]=[0]$ (see \cite[Proposition 6.2]{WM}).

With the aid of these claims, we can show that the Bers Schwarzian derivative map $\Phi$ 
is a holomorphic split submersion onto its image. 

\begin{theorem}\label{submersion}
The Bers Schwarzian derivative map $\Phi:M_{\sharp}^X(\mathbb{D})\to B_{\sharp}^X(\mathbb{D}^*)$ 
is a holomorphic split submersion onto its image $\Phi(M_{\sharp}^X(\mathbb{D}))=\beta(T)\cap B_{\sharp}^X(\mathbb{D}^*)$. 
\end{theorem}

\begin{proof}
Since $\Phi: M(\mathbb{D})\to B(\mathbb{D^*})$
is holomorphic and since $M_\sharp^X(\mathbb{D})$ and $B_\sharp^X(\mathbb{D}^*)$ are closed subspaces in the relative topology,
$\Phi:M_\sharp^X(\mathbb{D})\to B_\sharp^X(\mathbb{D}^*)$ is also holomorphic.
It remains to show that $\Phi$ is a split submersion onto its image $\Phi(M_\sharp^X(\mathbb{D}))$.
This is equivalent to showing that for every $\mu \in M_\sharp^X(\mathbb{D})$,
there is a holomorphic map $\sigma:U_\phi \to M_\sharp^X(\mathbb{D})$ defined on some neighborhood $U_\phi \subset 
\Phi(M_\sharp^X(\mathbb{D}))$ of $\phi=\Phi(\mu)$ such that $\sigma(\phi)=\mu$ and $\Phi \circ \sigma={\rm id}_{U_\phi}$.
The existence of some local holomorphic section can be given by a standard argument below.
 
We first complete showing that $\Phi$ is a split submersion by
assuming that there is a local holomorphic section $\sigma:U_\phi \to M_\sharp^X(\mathbb{D})$ at $\phi=\Phi(\mu)$.
We set $\nu=\mu^{-1} \ast \sigma(\phi)$, which belongs to $M_\sharp^X(\mathbb{D})$. We see that $r_\nu$ is a biholomorphic automorphism of $M_\sharp^X(\mathbb{D})$ which satisfies $\pi \circ r_\nu=\pi$ and
$r_\nu(\sigma(\phi))=\mu$.
Then, we obtain the required local section $r_\nu \circ \sigma$ on $U_\phi$ passing through $\mu$.

In the rest of the proof, we show the existence of a local holomorphic section. 
Let $\phi = \Phi(\mu)$ for a given $\mu \in M_\sharp^X(\mathbb{D})$. 
Without loss of generality, we may assume that $\mu = s([\mu])$, that is, $f^{\mu}$ is the barycentric extension of $f^{\mu}|_{\mathbb{S}}$. Here, $s: T \to M(\mathbb{D})$ is the barycentric section 
which maps $T_\sharp^X$ into $M_\sharp^X(\mathbb{D})$ by Corollary \ref{DE}. For
the quasiconformal homeomorphism $f_\phi=f_{\mu}:\widehat {\mathbb{C}} \to \widehat {\mathbb{C}}$ that is conformal on $\mathbb{D}^*$,
we set
$D=f_\phi(\mathbb{D})$, $D^* = f_\phi(\mathbb{D}^*)$, and $\gamma = f_{\phi}\circ j \circ f_{\phi}^{-1}$
for the reflection $j:\zeta \mapsto \zeta^*$ with respect to $\mathbb{S}$. 
We may assume that $f_{\phi}$ is normalized so that $\lim_{z \to \infty}(f_\phi(z)-z)=0$.
Since the barycentric extension $f^\mu$ is a bi-Lipschitz diffeomorphism with respect to the hyperbolic metric,
we see that so is $f_{\phi}|_{\mathbb{D}}$, and hence, the quasiconformal reflection $\gamma: D \to D^*$ is
a bi-Lipschitz diffeomorphism with respect to the hyperbolic metrics on $D$ and $D^*$.

Ahlfors \cite{Ah63} (see also \cite{GL, Le})
showed that there exists a constant $C_1 \geqslant 1$ depending only 
on $\Vert\mu\Vert_{\infty}$ such that 
\begin{equation}\label{1}
\frac{1}{C_1}\leqslant |\gamma(z)-z|^2\rho_{D^*}^{-2}(\gamma(z))|\bar{\partial}\gamma(z)|\leqslant C_1
\end{equation}
for every $z\in D$, where $\rho_{D^*}(z)$ is the hyperbolic density on $D^*$. 
We set 
$$
B_{\varepsilon}(\phi)=\{\psi \in B_\sharp^X(\mathbb{D}^*) \mid \Vert\psi - \phi\Vert_{B}< \varepsilon  \}
$$ 
for $\varepsilon > 0$. 
For each $\psi \in B_{\varepsilon}(\phi)$, 
there exists a unique locally univalent holomorphic function $f_{\psi}$ on $\mathbb{D}^*$ 
with the normalization as above such that 
$\mathcal{S}(f_{\psi})= \psi$. Let $g_{\psi} = f_{\psi}\circ f_{\phi}^{-1}|_{D^*}$. 
Then, we have that $\mathcal{S}(g_{\psi})\circ f_\phi (f'_\phi)^2 = \psi - \phi$ and 
$\sup_{z^*\in D^*}\rho_{D^*}^{-2}(z^*)|\mathcal{S}(g_{\psi})(z^*)|=\Vert\psi - \phi\Vert_{B}$.

When $\varepsilon>0$ is sufficiently small, it was proved in \cite{Ah63}
that $g_{\psi}$ is univalent (conformal) and can be extended to a quasiconformal homeomorphism of
$\widehat {\mathbb{C}}$ whose complex dilatation $\mu_{\psi}$ on $D$ has the form 
$$
%\begin{equation}\label{2}
 \mu_{\psi}(z)=\frac{\mathcal{S}(g_{\psi})(\gamma(z))(\gamma(z) - z)^2\bar{\partial}\gamma(z)}{2 + \mathcal{S}(g_{\psi})(\gamma(z))(\gamma(z) - z)^2\partial \gamma(z)}.
%\end{equation}
$$
We set $U_\phi=B_\varepsilon(\phi)$ for this $\varepsilon>0$.
Then by (\ref{1}), every $\psi \in U_\phi$ satisfies
\begin{equation}\label{3}
|\mu_{\psi}(z)|\leqslant C_2 |\mathcal{S}(g_{\psi})(\gamma(z))|\rho_{D^*}^{-2}(\gamma(z)) \qquad (z\in D)    
\end{equation}
for some constant $C_2>0$, which also depends only on $\Vert\mu\Vert_{\infty}$.

Consequently, $f_{\psi}=g_{\psi}\circ f_\phi$ is conformal on $\mathbb{D}^*$ 
and has a quasiconformal extension to $\widehat {\mathbb{C}}$ whose complex dilatation $\nu_{\psi}$ on $\mathbb{D}$ is given as
\begin{equation}\label{4}
    \nu_{\psi}=\frac{\mu+(\mu_{\psi}\circ f_\phi)\tau}{1+\bar{\mu}(\mu_{\psi}\circ f_\phi)\tau}, 
    \qquad \tau = \frac{\overline{\partial f_\phi}}{\partial f_\phi}. 
\end{equation}
It is well known that $\nu_{\psi}$ depends holomorphically on $\psi$. Now it follows from (\ref{3}) that 
\begin{equation*}
    \begin{split}
        |\mu_{\psi}(f_\phi(\zeta))|&\leqslant C_2 |\mathcal{S}(g_{\psi})(\gamma(f_\phi(\zeta)))|\rho_{D^*}^{-2}(\gamma(f_\phi(\zeta)))\\
        & = C_2 |\mathcal{S}(g_{\psi})(f_\phi(j(\zeta)))|\rho_{D^*}^{-2}(f_\phi(j(\zeta)))\\
        & = C_2 |\psi(j(\zeta))-\phi(j(\zeta))|\rho_{\mathbb{D}^*}^{-2}(j(\zeta))\\
        & = C_2 |\psi(\zeta^*)-\phi(\zeta^*)|\rho_{\mathbb{D}^*}^{-2}(\zeta^*)\\
    \end{split}
\end{equation*}
for every $\zeta \in \mathbb{D}$ with $\zeta^*=j(\zeta) \in \mathbb{D}^*$.

Since $\psi,\, \phi \in B_\sharp^X(\mathbb{D}^*)$,
the above estimate implies that $\mu_{\psi}\circ f \in M_\sharp^X(\mathbb{D})$.
Then, we see from (\ref{4}) that $\nu_{\psi} \in M_\sharp^X(\mathbb{D})$. 
Since $\Phi(\nu_{\psi})=\psi$, we conclude that $\sigma: U_\phi \to M_{\sharp}^{X}(\mathbb{D})$ 
defined by $\sigma(\psi)=\nu_\psi$
is a local holomorphic section to $\Phi$. This completes the proof. 
\end{proof}

\begin{corollary}\label{complexstructure}
The Bers embedding
$\beta: T_{\sharp}^X \to B_{\sharp}^X(\mathbb{D}^*)$ is a homeomorphism onto the domain 
$\beta(T) \cap B_{\sharp}^X(\mathbb{D}^*)$ in $B_{\sharp}^X(\mathbb{D}^*)$. Hence,
the Teich\-m\"ul\-ler space $T_{\sharp}^X$ has the complex structure modeled on the complex Banach space $B_{\sharp}^X(\mathbb{D}^*)$.
Under this complex structure, 
the projection $\pi:M_{\sharp}^X(\mathbb{D}) \to T_{\sharp}^X$ is also a holomorphic split submersion.
\end{corollary}

\begin{proof}
By the continuity of $\Phi:M_{\sharp}^X(\mathbb{D}) \to B_{\sharp}^X(\mathbb{D}^*)$, we see that $\beta$ is continuous.
For the other direction, the existence of the local continuous section to $\Phi$
shown in Theorem \ref{submersion} together with the continuity of the projection $\pi$
ensures the continuity of the inverse $\beta^{-1}:\beta(T) \cap B_{\sharp}^X(\mathbb{D}^*) \to T_{\sharp}^X$.
These facts prove that $\beta$ is a homeomorphism onto the image.
\end{proof}

Finally, we note that the corresponding result to Proposition \ref{L} is also valid for
the space of the holomorphic quadratic differentials.

\begin{proposition}\label{B}
For $X = \{\xi_1,\dots, \xi_n\} \subset \mathbb{S}$, 
$B_{\sharp}^X(\mathbb{D}^*) = B_{\sharp}^{\xi_1}(\mathbb{D}^*)+\dots + B_{\sharp}^{\xi_n}(\mathbb{D}^*)$.
\end{proposition}
\begin{proof}
For the Bers Schwarzian derivative map $\Phi: M_{\sharp}^{X}(\mathbb{D}) \to B_{\sharp}^X(\mathbb{D}^*)$, we consider 
its derivative $d_0\Phi: L_{\sharp}^{X}(\mathbb{D})\to B_{\sharp}^{X}(\mathbb{D}^*)$ at $0 \in M_{\sharp}^{X}(\mathbb{D})$. By Proposition \ref{L}, $L_{\sharp}^X(\mathbb{D}) = L_*^{\xi_1}(\mathbb{D})+\dots +L_{\sharp}^{\xi_n}(\mathbb{D})$. Since $d_0\Phi$ is a linear map, we see that 
\begin{equation*}
    \begin{split}
      d_0\Phi(L_{\sharp}^X(\mathbb{D})) & = d_0\Phi(L_{\sharp}^{\xi_1}(\mathbb{D})+\dots +L_{\sharp}^{\xi_n}(\mathbb{D})) \\
     & = d_0\Phi(L_{\sharp}^{\xi_1}(\mathbb{D}))+\dots +d_0\Phi(L_{\sharp}^{\xi_n}(\mathbb{D})) \\
     & = B_{\sharp}^{\xi_1}(\mathbb{D}^*)+\dots+B_{\sharp}^{\xi_n}(\mathbb{D}^*).\\
    \end{split}
\end{equation*}
Since $\Phi$ is a submersion by Theorem \ref{submersion}, $d_0\Phi: L_{\sharp}^{X}(\mathbb{D})\to B_{\sharp}^{X}(\mathbb{D}^*)$ is surjective, 
namely, $d_0\Phi(L_{\sharp}^X(\mathbb{D}^*)) = B_{\sharp}^X(\mathbb{D}^*)$. This completes the proof.
\end{proof}

\section{Quotient Teich\-m\"ul\-ler spaces and quotient Bers embedding}

We consider two quotients:
(1) $T_0 \backslash T_{\sharp}^X$; (2) $T_{\sharp}^X \backslash T$.
In both cases, we will prove that the quotient Bers embedding is well-defined and injective.
Moreover, by verifying that it is a homeomorphism onto the image in the quotient Banach space,
we provide a complex Banach manifold structure for each of them. For (1), this is essentially given by the
theory of the asymptotic Teich\-m\"ul\-ler space $AT=T_0 \backslash T$. For (2), we use Theorem \ref{EMS}
to introduce the equivalence relation that defines the quotient.

It is known that $T_0$ is a subgroup of $T$ under the operation $\ast$ given by
$[\mu] \ast [\nu]=[\mu \ast \nu]$ for any $[\mu], [\nu] \in T$. Then, the right coset 
$T_0 \backslash T$ is defined as the asymptotic Teich\-m\"ul\-ler space $AT$.
Let $p:T \to AT$ be the quotient projection and $P:B(\mathbb D^*) \to B_0(\mathbb D^*) \backslash B(\mathbb D^*)$
the projection onto the quotient Banach space $B_0(\mathbb D^*) \backslash B(\mathbb D^*)$.
By the results in \cite{GS}, the Bers embedding $\beta:T \to B(\mathbb D^*)$ is projected down to
a well-defined map $\hat \beta:AT \to B_0(\mathbb D^*) \backslash B(\mathbb D^*)$ by
$\hat \beta=P \circ \beta \circ p^{-1}$ and it is a local homeomorphism onto the image.
We call this map $\hat \beta$ the quotient Bers embedding.
Later, it was proved that $\hat \beta$ is in fact a global homeomorphism onto the image
(see \cite{GL} and \cite{EMS}).

As $T_{\sharp}^X$ is a closed subspace of $T$, the quotient of $T_{\sharp}^X$ by $T_0$ is defined 
to be $T_0 \backslash T_{\sharp}^X=p(T_{\sharp}^X)$ as a closed subspace of $AT$.
As $B_{\sharp}^X(\mathbb D^*)$ is also a closed subspace of $B(\mathbb D^*)$,
we immediately see the following:

\begin{proposition}\label{AT}
The restriction of $\hat \beta:AT \to B_0(\mathbb D^*) \backslash B(\mathbb D^*)$ to 
$T_0 \backslash T_{\sharp}^X$ defines a homeo\-morphism 
$\hat \beta:T_0 \backslash T_{\sharp}^X \to B_0(\mathbb D^*) \backslash B_{\sharp}^X(\mathbb D^*)$
onto the image $\hat \beta(AT) \cap (B_0(\mathbb D^*) \backslash B_{\sharp}^X(\mathbb D^*))$.
Hence, $T_0 \backslash T_{\sharp}^X$ is endowed with the complex structure modeled on $B_0(\mathbb D^*) \backslash B_{\sharp}^X(\mathbb D^*)$.
\end{proposition}

Concerning (2), we first introduce the following equivalence relation
with respect to $T_{\sharp}^X$:
$[\mu]$ and $[\nu]$ in $T$ are equivalent by definition if
$[\mu] \ast [\nu]^{-1} \in T_{\sharp}^Y$ for $Y=f^\nu(X)$.
We denote this equivalence by $[\mu] \sim_X [\nu]$.
It is easy to check that $\sim_X$ is an equivalence relation in $T$.
Then, denoting the set of all equivalence classes by $T/\!\sim_X$, we define the quotient projection by
$p_X:T \to T/\!\sim_X$.
On the contrary, as $B_{\sharp}^X(\mathbb D^*)$ is a closed subspace of $B(\mathbb D^*)$,
the projection $P_X:B(\mathbb D^*) \to B_{\sharp}^X(\mathbb D^*) \backslash B(\mathbb D^*)$ onto the quotient Banach space is
given as usual. Under these circumstances, we can obtain the following result:

\begin{theorem}\label{quotientBers}
The quotient Bers embedding $\hat \beta_X:T/\!\sim_X \to B_{\sharp}^X(\mathbb D^*) \backslash B(\mathbb D^*)$
is well-defined by $\hat \beta_X=P_X \circ \beta \circ p_X^{-1}$, and it is a homeomorphism onto the image.
Hence, $T/\!\sim_X$ is endowed with the complex structure modeled on $B_{\sharp}^X(\mathbb D^*) \backslash B(\mathbb D^*)$.
\end{theorem}

\begin{proof}
The well-definedness of $\hat \beta_X$ is shown as follows. 
Suppose that $p_X([\mu])=p_X([\nu])$ for $[\mu], [\nu] \in T$, that is,
$[\mu] \sim_X [\nu]$.
By definition, $[\mu] \ast [\nu]^{-1} \in T_{\sharp}^Y$ for $Y=f^\nu(X)$.
%and it follows that
%$\pi^{-1}([\mu] \ast [\nu]^{-1}) \in M_{\sharp}^Y(\mathbb D)$. Moreover, $s([\mu]) \ast s([\nu]^{-1})
%\in M_{\sharp}^Y(\mathbb D)$. Then, 
By the implication $(4) \Rightarrow (1)$ in Theorem \ref{EMS},
we have $\Phi(\mu)-\Phi(\nu) \in B_{\sharp}^X(\mathbb D^*)$. By $\Phi(\mu)=\beta([\mu])$ and 
$\Phi(\nu)=\beta([\nu])$, we see that $P_X \circ \beta([\mu])=P_X \circ \beta([\nu])$, which shows that
$\hat \beta_X$ is well-defined.
The injectivity of $\hat \beta_X$ similarly follows from the implication $(1) \Rightarrow (4)$ in Theorem \ref{EMS}.

Since $p_X$ is the quotient map and $\beta$ is continuous, we see that $\hat \beta_X$ is continuous.
The continuity of $\hat \beta^{-1}_X$ as well as the claim that the image of $\hat \beta$ is open 
in $B_{\sharp}^X(\mathbb D^*) \backslash B(\mathbb D^*)$ follows from the fact that $P_X$ is an open map.
\end{proof}

%For any $n \in \mathbb N$, we consider the union
%$T_{\sharp}^n=\bigcup_{|X|=n} T_{\sharp}^X$ taken over all subset $X \subset \mathbb S$ whose cardinality is $n$.
%
%\medskip
%\noindent
%{\bf Problems.}
%(1) Is the quotient space $T_0 \backslash T_{\sharp}^n$ a submanifold of
%$AT=T_0 \backslash T$? (2) Does $T_0 \backslash T_{\sharp}^n \subset AT$ admit a natural group structure? 

\section{rigidity theorems}

Let $Q$ and $Q'$ be groups consisting of sense-preserving self-homeomorphisms of $\mathbb S$ with 
$\mbox{\rm M\"ob}(\mathbb{S}) \subset Q \subsetneqq Q'$ in general. 
Rigidity of $Q$ with respect to $Q'$ is a property that
for a subgroup $G$ of $\mbox{\rm M\"ob}(\mathbb{S})$ and for $f \in Q'$,
the condition $fGf^{-1} \subset Q$ implies that $f \in Q$. Namely,
the representation of $G$ in $Q$ given by the conjugation of an element of $Q'$ is
only an inner automorphism of $Q$. 

We consider this rigidity for a subgroup $Q$ of $\rm QS$ with certain regularity.
In concrete, we set $Q={\rm Diff}^{r}_+(\mathbb{S})$, the group consisting of all sense-preserving
$C^r$-diffeomorphisms $f$ of $\mathbb S$ onto itself for any real number $r>1$. When $r \notin \mathbb N$,
this means that $f$ is in $C^{[r]}$ and its $[r]$-th derivative is H\"older continuous of exponent $r-[r]$.
We also need to restrict a subgroup $G$ of $\mbox{\rm M\"ob}(\mathbb{S}) \cong \mbox{\rm M\"ob}(\mathbb{D})$ to be non-elementary,
which means by definition that there is no finite subset $X \subset \overline{\mathbb D}$ that satisfies $g(X)=X$ for
every $g \in G$.

The following results were proved in \cite{Ma16}. Statement (2) has been slightly generalized by
applying a result in \cite{Ma19}. 

\begin{proposition}
Let $G$ be a non-elementary subgroup of $\mbox{\rm M\"ob}(\mathbb{S})$.
$(1)$ If $fGf^{-1} \in \mbox{\rm M\"ob}(\mathbb{S})$ for $f \in {\rm Sym}$, then $f \in \mbox{\rm M\"ob}(\mathbb{S})$.
$(2)$ If $fGf^{-1} \in {\rm Diff}^{r}_+(\mathbb{S})$ $(r>1)$ for $f \in {\rm Sym}$, then $f \in {\rm Diff}^{r}_+(\mathbb{S})$.
\end{proposition}

We will generalize $Q'=\rm Sym$ to ${\rm QS}_{\sharp}^X$ and prove the following two theorems corresponding to
(1) and (2) in the above proposition.

\begin{theorem}
Let $G$ be a non-elementary subgroup of $\mbox{\rm M\"ob}(\mathbb{S})$.
If $fGf^{-1} \in \mbox{\rm M\"ob}(\mathbb{S})$ for $f \in {\rm QS}_{\sharp}^X$ with a finite subset $X \subset \mathbb S$,
then $f \in \mbox{\rm M\"ob}(\mathbb{S})$.
\end{theorem}

\begin{proof}
The equivalence class of $f \in {\rm QS}_{\sharp}^X$ defines the element 
$[f]$ in $T_{\sharp}^X=\mbox{\rm M\"ob}(\mathbb{S}) \backslash {\rm QS}_{\sharp}^X$. We set
$\varphi=\beta([f]) \in \Phi(M_{\sharp}^X(\mathbb D))$, which belongs to $B_{\sharp}^X({\mathbb D}^*)$ by Theorem \ref{Bers}.
The condition $f G f^{-1} \subset \mbox{\rm M\"ob}(\mathbb{S})$ is equivalent to that
$$
g^*\varphi(z)=\varphi(g(z))g'(z)^2=\varphi(z) \quad (z \in \mathbb D^*)
$$
for every $g \in G \subset \mbox{\rm M\"ob}({\mathbb S}) \cong \mbox{\rm M\"ob}({\mathbb D}^*)$,
the group of all M\"obius transformations of $\mathbb{D}^*$.
Then,
$$
\rho_{{\mathbb D}^*}^{-2}(z)|\varphi(z)|=\rho_{{\mathbb D}^*}^{-2}(z)|g^*\varphi(z)|
=\rho_{{\mathbb D}^*}^{-2}(g z)|\varphi(g z)|.
$$
Since $G$ is non-elementary, we can choose a hyperbolic element $g_0 \in G$ whose attracting fixed point 
is not in $X$. 
By $\varphi \in B_{\sharp}^X({\mathbb D}^*)$, we have $\varphi(g_0^n(z)) \to 0$ $(n \to \infty)$ for all $z \in {\mathbb D}^*$.
Hence, $\varphi(z) \equiv 0$. This means that
$[f]=[{\rm id}]$, and equivalently, $f \in \mbox{\rm M\"ob}(\mathbb{S})$.
\end{proof}

\begin{theorem}
Let $G$ be a non-elementary subgroup of $\mbox{\rm M\"ob}(\mathbb{S})$.
If $fGf^{-1} \in {\rm Diff}^{r}_+(\mathbb{S})$ $(r>1)$ for $f \in {\rm QS}_{\sharp}^X$ with a finite subset $X \subset \mathbb S$,
then $f \in {\rm Diff}^{r}_+(\mathbb{S})$.
\end{theorem}

\begin{proof}
As before, we consider 
$\varphi=\beta([f])$ in $B_{\sharp}^X({\mathbb D}^*)$.
We also choose a hyperbolic element $g_0 \in G$ whose fixed points 
are not in $X$. By \cite[Proposition 4.3]{Ma16},
the condition $f G f^{-1} \subset {\rm Diff}^{r}_+(\mathbb{S})$ implies that
$\psi=g_0^*\varphi-\varphi$ belongs to the Banach space
$$
B_0^\alpha(\mathbb D^*)=\{\varphi \in B({\mathbb D}^*) \mid \sup_{z \in {\mathbb D}^*} \rho_{\mathbb D^*}^{\alpha-2}(z)|\varphi(z)|<\infty\}
$$
for some $\alpha \in (0,1)$. 

As in \cite[Proposition 4.4]{Ma16}, we can show that 
\begin{equation}\label{abel}
\varphi(z)=-\sum_{i=0}^{\infty} (g_0^*)^i \psi(z)=\sum_{i=1}^{\infty} (g_0^*)^{-i}\psi(z)
\end{equation} 
for each $z \in \mathbb D^*$ as follows.
For each $i \in \mathbb Z$, it holds $(g_0^*)^i \psi=(g_0^*)^{i+1} \varphi-(g_0^*)^{i}\varphi$.
Summing up this from $i=0$ to $n \geqslant 0$, we have
$$
\sum_{i=0}^{n} (g_0^*)^i \psi=(g_0^*)^{n+1} \varphi-\varphi.
$$
Here, $\lim_{n \to +\infty}(g_0^*)^{n+1} \varphi(z)=0$. Indeed, for each $z \in {\mathbb D}^*$,
$$
\rho_{\mathbb D^*}^{-2}(z)|(g_0^*)^{n+1}\varphi(z)|
=\rho_{\mathbb D^*}^{-2}(g_0^{n+1}(z))|\varphi(g_0^{n+1}(z))|,
$$
and the right side term converges to $0$ as $n \to \infty$ because $\varphi \in B_{\sharp}^X(\mathbb D^*)$.
Thus, $\varphi(z)=-\sum_{i=0}^{\infty} (g_0^*)^i \psi(z)$.
If we sum up the above equation from $i=-1$ to $-n \leqslant -1$ and take the limit as $n \to \infty$, 
then we can obtain the second equation in
the same reason.

From Formulae (\ref{abel}), we can prove that $\varphi$ itself belongs to $B_0^\alpha(\mathbb D^*)$ (\cite[Lemma 4.5]{Ma16}).
This in particular implies that $f$ is a diffeomorphism of $\mathbb S$ onto itself.
Then, by \cite[Theorem 7.4]{Ma19}, we can conclude that $f \in {\rm Diff}^{r}_+(\mathbb{S})$. 
\end{proof}

\section{Exhaustion by countable sequences}

In this section, we investigate the relationship between the exhaustion of a subset $X \subset \mathbb S$ by
an increasing sequence $X_n \nearrow X$ and the inclusion $\bigcup_{n=1}^\infty T_{\sharp}^{X_n} \subset T_{\sharp}^{X}$
of the corresponding Teichm\"uller spaces. We start with a basic lemma.

\begin{lemma}\label{strict}
If $X \subsetneqq X' \subset \mathbb S$, then the inclusion $T_{\sharp}^{X}\subset T_{\sharp}^{X'}$ is strict.
\end{lemma}
\begin{proof}
The Bers embeddings satisfy $\beta(T_{\sharp}^{X})=\beta(T) \cap B_{\sharp}^{X}(\mathbb{D}^*)$ and
$\beta(T_{\sharp}^{X'})=\beta(T) \cap B_{\sharp}^{X'}(\mathbb{D}^*)$, where
$\beta(T)$ is an open subset of $B(\mathbb{D}^*)$.
Hence, we have only to show that $B_{\sharp}^{X}(\mathbb{D}^*) \subset B_{\sharp}^{X'}(\mathbb{D}^*)$ is a strict inclusion. 
This can be done by showing some element of $B_{\sharp}^{X'}(\mathbb{D}^*)$ does not belong to $B_{\sharp}^{X}(\mathbb{D}^*)$. 

We choose some $\xi \in X' \setminus X$.
For a parabolic transformation $\gamma \in \mbox{\rm M\"ob}(\mathbb{S}) \cong \mbox{\rm M\"ob}(\mathbb{D}^*)$ with the fixed point $\xi$,
we consider the $\langle \gamma \rangle$-invariant subspace
$$
B_{\sharp}^{\xi}(\mathbb{D}^*, \langle \gamma \rangle)= \{\psi\in B_{\sharp}^{\xi}(\mathbb{D}^*) \mid \gamma^*\psi= \psi\}
\subset B_{\sharp}^{X'}(\mathbb{D}^*),  
$$
which contains a non-zero element. 
For any non-zero element $\varphi \in B_{\sharp}^{\xi}(\mathbb{D}^*, \langle \gamma \rangle)$, 
there exists some $z_0 \in \mathbb{D}^*$ such that $\rho_{\mathbb{D}^*}^{-2}(z_0)|\varphi(z_0)| \neq 0$. 
Moreover, $\rho_{\mathbb{D}^*}^{-2}(z)|\varphi(z)|$ takes the same value on
the orbit $\{\gamma^n(z_0)\}_{n \in \mathbb Z}$ of $z_0$, which accumulates to $\xi \notin X$. 
This implies that $\varphi$ does not belong to $B_{\sharp}^{X}(\mathbb{D}^*)$.
\end{proof}

For any strictly increasing infinite sequence of subsets 
$X_1 \subsetneqq X_2 \subsetneqq \cdots \subsetneqq X_n \subsetneqq \cdots$ of $\mathbb S$,
we consider the sequence of the corresponding Teichm\"uller spaces 
$$
T_{\sharp}^{X_1} \subset T_{\sharp}^{X_2} \subset \cdots \subset T_{\sharp}^{X_n} \subset \cdots.
$$ 
By Lemma \ref{strict}, these inclusions are all strict.
We compare $T_{\sharp}^{X}$ with $\bigcup_{n=1}^{\infty}T_{\sharp}^{X_n}$ where $X=\bigcup_{n=1}^\infty X_n \subset \mathbb S$.
%We remark that $\bigcup_{n=1}^{\infty}T_{\sharp}^{X_n}$ does not depend on the order of the sequence $X = \{\xi_1, \xi_2, \dots \}$.

\begin{theorem}
Under the above circumstances,
%the following are satisfied:
%\begin{enumerate}
%\item
the increasing union $\bigcup_{n=1}^{\infty}T_{\sharp}^{X_n}$ is not closed in $T_{\sharp}^X$, and hence,
$\bigcup_{n=1}^{\infty}T_{\sharp}^{X_n}$ is strictly contained in $T_{\sharp}^X$.
%\item
%The Teich\-m\"ul\-ler space $T_{\sharp}^X$ is strictly contained in $T$ even if $X$ is dense in $\mathbb S$.
%\end{enumerate}
\end{theorem}

\begin{proof}
As the Bers embeddings satisfy $\beta(T_{\sharp}^{X_n})=\beta(T) \cap B_{\sharp}^{X_n}(\mathbb{D}^*)$ (also holds for $T_{\sharp}^{X}$) and
$\beta(T)$ is an open subset of $B(\mathbb{D}^*)$, it suffices to consider the problems for the closed subspaces
$B_{\sharp}^{X_n}(\mathbb{D}^*)$ and $B_{\sharp}^{X}(\mathbb{D}^*)$.
To prove this claim, we use a corollary to the Baire category theorem (see \cite[p.10]{GT}). 
It asserts that  
for a sequence of nowhere dense subsets $\{E_n\}_{n=1}^{\infty}$ of a complete metric space in general, the countable union
$\bigcup_{n=1}^{\infty}E_n$ has empty interior. 

Assume that $\bigcup_{n=1}^{\infty}B_{\sharp}^{X_n}$ is closed. Then, $\bigcup_{n=1}^{\infty}B_{\sharp}^{X_n}$ is a complete metric space with respect to the induced metric by the norm $\Vert\cdot\Vert_{\infty}$. It is obvious that the origin is an interior point of $\bigcup_{n=1}^{\infty}B_{\sharp}^{X_n}$ in the relative topology. Thus, combining this with the fact that $B_{\sharp}^{X_n}$ is closed for every $n$, we can conclude that there is some $n_0$ such that $B_{\sharp}^{X_{n_0}}(\mathbb{D}^*)$ has an interior point in the relative topology of 
$\bigcup_{n=1}^{\infty}B_{\sharp}^{X_n}(\mathbb{D}^*)$. 
This means that there is some open subset $U$ of $B(\mathbb{D}^*)$ such that 
$$
U \cap B_{\sharp}^{X_{n_0}}(\mathbb{D}^*) =U \cap \bigcup_{n=1}^{\infty}B_{\sharp}^{X_n}(\mathbb{D}^*).
$$
This occurs only when $B_{\sharp}^{X_n}(\mathbb{D}^*)$ are same for all $n \geqslant n_0$.
However, this contradicts Lemma \ref{strict} and its proof.
%(2) We take $\varphi \in B(\mathbb{D}^*)$ such that $\rho_{\mathbb{D}^*}^{-2}(z)|\varphi(z)|$ never vanishes at any boundary point of $\mathbb S$. For example, a non-zero $\Gamma$-invariant element for a cocompact Fuchsian group $\Gamma$ satisfies this condition. Then, the distance in the norm $\Vert\cdot\Vert_{B}$ from $\varphi$ to $B_{\sharp}^{X}(\mathbb D^*)$ is the same as $\Vert \varphi \Vert_B$,
%which implies that there exists a neighborhood of $\varphi$ that does not intersect $B_{\sharp}^{X}(\mathbb D^*)$.
%Thus, $B_{\sharp}^{X}(\mathbb D^*) \subsetneqq B(\mathbb{D}^*)$.
\end{proof}

This theorem in particular shows that even in the case of $X=\mathbb S$, 
the increasing union $\bigcup_{n=1}^{\infty}T_{\sharp}^{X_n}$ does not give an exhaustion of the
universal Teichm\"uller space $T=T_{\sharp}^{\mathbb S}$.

We apply the above results to an ordered infinite sequence $X=\{\xi_1, \xi_2, \dots \}$ of distinct points on $\mathbb{S}$. 
We denote the set of the first $n$-th points by $X_n = \{\xi_1, \xi_2, \dots, \xi_n\}$. 
Then, we have a strictly increasing sequence of piecewise symmetric Teich\-m\"ul\-ler spaces $\{T_{\sharp}^{X_n}\}_{n=1}^\infty$.
We are interested in the case where $X$ is dense in $\mathbb S$.
However, the above results imply that there is no exhaustion of $T$ by an increasing sequence of
the piecewise symmetric Teich\-m\"ul\-ler spaces. In fact, we see more:

\begin{proposition}
$(1)$ For any $X=\{\xi_1, \xi_2, \dots \}$, the closure $\overline{\bigcup_{n=1}^{\infty}T_{\sharp}^{X_n}}$ is strictly contained in $T$.
$(2)$ If $X=\{\xi_1, \xi_2, \dots \}$ has no accumulation point in $X$, then 
$\overline{\bigcup_{n=1}^{\infty}T_{\sharp}^{X_n}}$ coincides with $T_{\sharp}^X$.
\end{proposition}

\begin{proof}
(1) It is clear that $\overline{\bigcup_{n=1}^{\infty}T_{\sharp}^{X_n}} \subset T_{\sharp}^X$. Since $X \subsetneqq \mathbb S$,
we have $T_{\sharp}^X \subsetneqq T$ by Lemma \ref{strict}.
(2) If $X$ has no accumulation point in $X$, then
any compact subset $K \subset {\mathbb D}\cup X$ has at most finitely many points in $X$.
For any $\mu \in M_{\sharp}^X(\mathbb D)$ and for any $\varepsilon>0$, there is a compact subset $K \subset {\mathbb D}\cup X$
such that $\Vert\mu|_{{\mathbb D} \setminus K}\Vert_\infty <\varepsilon$. Let $\mu_K=\mu \cdot 1_K$. Then, $\mu_K$ belongs to
some $M_{\sharp}^{X_n}(\mathbb D)$ and $\Vert \mu-\mu_K \Vert <\varepsilon$. This implies that 
$M_{\sharp}^X(\mathbb D) \subset \overline{\bigcup_{n=1}^{\infty}M_{\sharp}^{X_n}(\mathbb D)}$.
From this, we see that $T_{\sharp}^X = \overline{\bigcup_{n=1}^{\infty}T_{\sharp}^{X_n}}$.
\end{proof}

%\begin{remark}
%{\rm
%From statement (1) above, we may ask a question  whether  
%$\overline{\bigcup_{n=1}^{\infty}T_{\sharp}^{X_n}}= T_{\sharp}^{X}$ or not. 
%This seems depend on the sequence $X = \{\xi_1, \xi_2, \dots \}$. If $X$ has no accumulation point in $X$, then
%the equality holds because any compact subset $K \subset {\mathbb D}\cup X$ has at most finitely many points in $X$ in this case.
%}
%\end{remark}

%The above results in particular implies that there is no exhaustion of $T$ by an increasing sequence of
%the piecewise symmetric Teich\-m\"ul\-ler spaces.
%However, we see that $T$ has a decomposition into ``partially'' symmetric Teich\-m\"ul\-ler spaces.
%
%\begin{proposition}
%Let $X_1$ and $X_2$ be open intervals in $\mathbb S$ such that $X_1 \cup X_2=\mathbb S$.
%Then, $B(\mathbb D^*)=B_{\sharp}^{X_1}(\mathbb D^*)+B_{\sharp}^{X_2}(\mathbb D^*)$.
%\end{proposition}

\end{document}